\documentclass{article}
\usepackage[english]{babel}

\usepackage[letterpaper,top=2cm,bottom=2cm,left=3cm,right=3cm,marginparwidth=1.75cm]{geometry}

\usepackage{amssymb}
\usepackage{amsmath}
\usepackage{amsthm}
\usepackage{ulem}
\usepackage{xcolor}

\usepackage{authblk}

\usepackage{changepage}
\usepackage{caption}
\usepackage{lipsum}

\usepackage{tkz-graph}
\tikzstyle{vertex}=[circle, draw, inner sep=3pt, minimum size=6pt]

\usepackage{lineno}
\usepackage{graphicx}
\usepackage[colorlinks=true, allcolors=blue]{hyperref}

\newtheorem{theorem}{Theorem}
\newtheorem{lemma}[theorem]{Lemma}

\newtheorem{conjecture}[theorem]{Conjecture}

\newtheorem{corollary}[theorem]{Corollary}

\theoremstyle{definition}

\newcommand{\Z}{\operatorname{Z}}

\title{Zero Forcing and Vertex Independence Number on Cubic and Subcubic Graphs}

\author[1]{Houston Schuerger}
\author[2]{Nathan Warnberg}
\author[3]{Michael Young}

\affil[1]{Department of Mathematics, University of Texas Permian Basin, \{schuerger\_h@utpb.edu\}}

\affil[2]{Department of Mathematics and Statistics, University of Wisconsin-La Crosse, \{nwarnberg@uwlax.edu\}}

\affil[3]{Department of Mathematical Sciences, Carnegie Mellon University,\{michaely@andrew.cmu.edu\}}

\begin{document}

\maketitle

\begin{abstract}

Motivated by a conjecture from the automated conjecturing program TxGraffiti, in this paper the relationship between the zero forcing number, $\Z(G)$, and the vertex independence number, $\alpha(G)$, of cubic and subcubic graphs is explored. TxGraffiti conjectures that for all connected cubic graphs $G$, that are not $K_4$, $\Z(G) \leq \alpha(G) + 1$. This work uses decycling partitions of upper-embeddable graphs to show that almost all cubic graphs satisfy $\Z(G) \leq \alpha(G) + 2$, provides an infinite family of cubic graphs where $\Z(G) = \alpha(G) + 1$, and extends known bounds to subcubic graphs.

\end{abstract}


\section{Introduction}

Zero forcing is a graph-theoretic concept that involves a dynamic process of coloring vertices in a graph. Initially, a subset of the vertices is colored blue, while the rest remain white. The primary rule of the zero forcing process is that a blue vertex can force a white vertex to turn blue if it is the only white neighbor of the blue vertex. This rule is applied iteratively until no further vertices can be forced to change color. The objective is to find the smallest set of initially blue vertices which can eventually turn all vertices blue.

Zero forcing has a number of practical applications, including linear algebra, where it was initially developed to address the minimum rank problem for graphs (see \cite{AIM}). 
In this context, the minimum rank of a graph is related to the rank of symmetric matrices whose non-zero entries correspond to the edges of the graph. Zero forcing is also used in studying network controllability in both classical and quantum systems (see \cite{Burgarth1, Burgarth2, Severini}). The ability to identify a minimal set of vertices that can ``control" the rest of the graph is key to understanding how information or influence can propagate through various types of networks.

Another prominent application of zero forcing is in the area of power domination, which models the problem of monitoring electrical grids (see \cite{dean, PMU}). In this model, a power monitoring unit (PMU) is placed on a vertex of the graph, and the objective is to minimize the number of PMUs required to monitor the entire network. 
This is closely related to the zero forcing process because once a vertex is monitored, neighboring vertices can also be observed under certain conditions, mirroring the propagation rules of zero forcing.

The study of zero forcing has also inspired various graph theoretic investigations such as studying the relationships between zero forcing and other graph parameters, such as the edge density, path cover number, chromatic number, spectral radius, vertex cover number, and domination number. 
These studies have explored how graph structure affects the zero forcing process, leading to insights into the complexity of the propagation process and its implications across disciplines.
This paper extends what is known about the relationship between zero forcing and vertex independence.

\subsection*{{\it TxGraffiti}}

The use of computer assistance to make conjectures and prove theorems in mathematics is powerful and growing in use. One such conjecturing program that has inspired the pursuit of many relationships between graph parameters is {\it TxGraffiti}. TxGraffiti is an automated conjecturing program that produces graph theoretic conjectures in the form of conjectured inequalities. Many conjectures produced by TxGraffiti have been proved (see \cite{BDSY,CDP,DHtotal,DH}).
More information concerning the history, methods, and contributions of TxGraffiti can be found in \cite{txgrafhist}.

 In \cite{CDP}, TxGraffiti conjectures that bound the independence number of cubic and claw-free graphs by the domination number of the graph are proved. The TxGraffitti conjectures that are proved in \cite{BDSY} relate the vertex cover number and the zero forcing number in claw-free graphs. This paper addresses a particular TxGraffitti conjecture that bounds the zero forcing number of a graph with the independence number of the graph.

\begin{conjecture}[\cite{BDSY}]\label{theconjecture}
If $G \neq K_4$ is a connected cubic graph, then $\Z(G) \le \alpha(G) + 1$.
\end{conjecture}

Previous work on this Conjecture \ref{theconjecture}, in \cite{DH}, shows that it holds if $G$ is a claw-free graph. 

\begin{theorem}[\cite{DH}] \label{thm:clawfree}
    If $G\neq K_4$ is a connected, claw-free, cubic graph, then  $\Z(G) \le \alpha(G) + 1$.
\end{theorem}

\subsection*{Notation and Terminology}

All graphs in this paper are simple, finite, and undirected. The maximum degree of a graph $G$ is denoted by $\Delta(G)$. A graph is {\it cubic} if it is $3$-regular and {\it subcubic} if it has maximum degree at most $3$.  Given set $S \subset V(G)$, $G-S$ is the subgraph of $G$ induced by $V(G) \setminus S$. Given $S\subseteq E(G)$, $G-S$ is the graph with all edges of $S$ deleted from $G$. A {\it claw} is the graph $K_{1,3}$ and a graph is {\it claw-free} if it does not contain a claw as an induced subgraph. A {\it claw center} of $G$ is a vertex of degree 3 in an induced claw of $G$.

A set of pairwise nonadjacent vertices in $G$ is an {\it independent set} of $G$. The number of vertices in a maximum independent set in $G$ is the {\it independence number} of $G$, denoted $\alpha(G)$. The set is called {\it near independent} if it induces exactly one edge.

In standard zero forcing, each vertex of a graph is colored blue or white and the following {\it color change rule} is applied: If a blue vertex $v$ has exactly one white neighbor $w$, then $v$ can {\it force} $w$ to become blue. 
A {\it standard zero forcing set} of a graph $G$ is a set $B \subseteq V(G)$ such that if each vertex in $B$ is initially colored blue, then all of the vertices of $G$ could be forced blue. 
The {\it standard zero forcing number} of $G$, denoted $\Z(G)$, is the cardinality of a minimum zero forcing set of $G$. 
Since this paper only discusses standard zero forcing, it will be referred to as zero forcing. 
For other types of zero forcing see \cite{VariantsZF} and \cite{ BookZF}.

In the zero forcing process exactly one white vertex is forced blue during each step, a list of these forces in the order in which they occur is known as a {\it chronological list of forces}. These forces are called {\it a set of forces} when they are unordered and color the entire vertex set of the graph blue. A {\it forcing chain} is a sequence of vertices $(v_1,v_2,\dots,v_k)$ such that for $1 \le i \le k-1$, $v_i$ forces $v_{i+1}$ in a set of forces.

A {\it fort} of graph $G$ is a nonempty set of vertices $F$ such that no vertex outside of $F$ is adjacent to exactly one vertex in $F$. 
If each vertex of $F$ is blue and each vertex of $G-F$ is white, then no forces can happen. Therefore, every zero forcing set for $G$ must contain at least one vertex of each fort of $G$.
Another well known result is that given a vertex $v$ in a connected, nontrivial graph $G$ it is always possible to find a minimum zero forcing set that does not contain $v$.

In Section \ref{sec:cubic}, a slightly weaker version of Conjecture \ref{theconjecture} is shown to be true for almost every cubic graph and an infinite family of graphs is provided to show that the conjectured bound is tight.  Section \ref{sec:subcubic} extends results about cubic graphs to subcubic graphs and provides a bound for subcubic graphs based on the number of claw centers.  The paper concludes with Section \ref{sec:conc} where relationships between the zero forcing number and the independence number are shown for {\it any} graph. 


\section{Bounds on Cubic Graphs}\label{sec:cubic}

In this section, Conjecture \ref{theconjecture} is confirmed for specific classes of cubic graphs and demonstrated to be tight for an infinite family of cubic graphs. Additionally, the developed techniques are applied to establish that a slightly weaker upper bound holds for nearly all cubic graphs.

\begin{lemma}\label{lem:allpaths}
    If $G$ is a cubic graph with no components isomorphic to $K_4$, then there exists a maximum independent set $A$ such that the connected components of $G-A$ are all paths.
\end{lemma}

\begin{proof} Let $A$ be a maximum independent set of $G$ such that the number of cycles in $H = G-A$ is minimized. The connected components of $H$ are paths or cycles since $\Delta(H) \le 2$. Assume $C$ is a cycle in $H$ with $c_0 \in V(C)$ and $a_0\in A$ such that $a_0$ is adjacent to $c_0$.

Let $c_0 a_0 c_1 a_1 \ldots c_k a_k$ be a longest path that starts at $c_0$, alternates between $H$ and $A$, has $deg_{H}(c_i) = 1$ for $0< i \le k$, no two vertices in $\{c_0, c_1, \ldots, c_k\}$ are in the same component of $H$, and $a_{i-1}$ is adjacent to the other degree 1 vertex in the component of $H$ containing $c_{i}$ for $0 < i \le k$ (see Figure \ref{fig:alternate}). 

The set of vertices $A'=(A \setminus \{a_0, a_1, \ldots, a_k\}) \cup \{c_0, c_1, \ldots, c_k\}$ is also an independent set of $G$. 
Since the number of cycles in $G-A'$ does not increase, it must be the same. 
In order for the number of cycles in $G-A'$ to be the same $a_k$ must be in a cycle in $G-A'$.  Therefore, the other two neighbors of $a_k$ must be in the same component of $H$ and at least one of those neighbors has degree 2 in $H$; otherwise, the path $c_0 a_0 c_1 a_1 \ldots c_k a_k$ is not the longest such path. If $x \neq c_k$ is a neighbor of $a_k$ with degree 2 in $G-A$, then $(A \setminus \{a_0, a_1, \ldots, a_k\}) \cup \{c_0, c_1, \ldots, c_k, x\}$ is a larger independent set unless $x$ is adjacent to some $c_i$.  
If $c_ix \in E(H)$ for some $0 \le i \le k$, then $A''=(A \setminus \{a_0, a_1, \ldots, a_{i-1}, a_k\}) \cup \{c_0, c_1, \ldots, c_{i-1}, x\}$ is also a maximum independent set and the number of cycles in $G-A''$ is smaller than the number of cycles in $H$. This contradiction implies that $H$ has no cycle.\end{proof}

\begin{figure}[h!]
    \centering
    \captionsetup{width=0.8\textwidth}
    \begin{tikzpicture}

        \draw (-3,0) ellipse (1cm and 1cm); 
        \node[circle,fill=black,inner sep=0pt,minimum size=4pt,label=below:{$a_0$}] (a0) at (-3,-2) {};
        \node[circle,fill=black,inner sep=0pt,minimum size=4pt,label=above:{$c_0$}] (c0) at (-3,-1) {};

        \node[circle,fill=black,inner sep=0pt,minimum size=4pt,label=above:{}] (c1a) at (-1,0) {};
        \node[circle,fill=black,inner sep=0pt,minimum size=4pt,label=above:{$c_1$}] (c1b) at (-0,0) {};
        \draw (-1,0) node {} -- (-0,0) [dashed] node {};
        \node[circle,fill=black,inner sep=0pt,minimum size=4pt,label=below:{$a_1$}] (a1) at (0,-2) {};

        \node[circle,fill=black,inner sep=0pt,minimum size=4pt,label=above:{}] (c2a) at (1,0) {};
        \node[circle,fill=black,inner sep=0pt,minimum size=4pt,label=above:{$c_2$}] (c2b) at (2,0) {};
        \draw (1,0) node {} -- (2,0) [dashed] node {};
        \node[circle,fill=black,inner sep=0pt,minimum size=4pt,label=below:{$a_2$}] (a2) at (2,-2) {};

         \node[circle,fill=black,inner sep=0pt,minimum size=2pt,label=above:{}] (e1) at (2.5,0) {};
         \node[circle,fill=black,inner sep=0pt,minimum size=2pt,label=above:{}] (e1) at (2.75,0) {};
         \node[circle,fill=black,inner sep=0pt,minimum size=2pt,label=above:{}] (e1) at (3,0) {};
         \node[circle,fill=black,inner sep=0pt,minimum size=4pt,label=below:{$a_{k-1}$}] (ak-1) at (3,-2) {};

         \node[circle,fill=black,inner sep=0pt,minimum size=4pt,label=above:{}] (cka) at (3.5,0) {};
        \node[circle,fill=black,inner sep=0pt,minimum size=4pt,label=above:{$c_k$}] (ckb) at (4.5,0) {};
        \draw (3.5,0) node {} -- (4.5,0) [dashed] node {};
        \node[circle,fill=black,inner sep=0pt,minimum size=4pt,label=below:{$a_k$}] (ak) at (4.5,-2) {};
        
	\Edge(c0)(a0)
        \Edge(c1a)(a0)
        \Edge(c1b)(a0)
        \Edge(a1)(c1b)
        \Edge(a1)(c2b)
        \Edge(a1)(c2a)
        \Edge(a2)(c2b)
        \Edge(ak-1)(cka)
        \Edge(ak-1)(ckb)
        \Edge(ak)(ckb)
	 	
	 \end{tikzpicture}
    \caption{A longest path $c_0a_0c_1a_1\dots c_ka_k$ that alternates between an independent set $A$ and components in $G-A$.}\label{fig:alternate}
\end{figure}
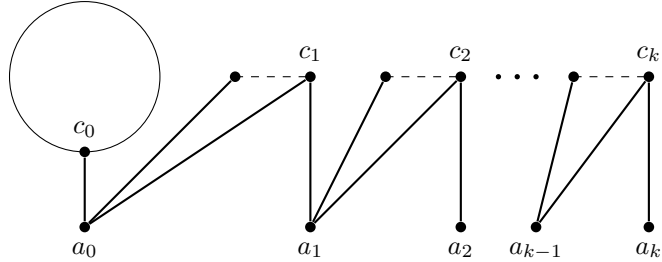

The {\it vertex cover number} of $G$, denoted $\beta(G)$, is the minimum number of vertices needed to cover every vertex of $G$ and the {\it edge cover number} of $G$, denoted $\overline{\beta}(G)$, is the minimum number of edges needed to cover every vertex of $G$. K\"onig's Theorem relates the vertex cover number and the size of a maximum matching in a bipartite graph. It is stated here in terms of $\overline{\beta}(G)$ and $\alpha(G)$. 

\begin{theorem}[\cite{konig1916}]\label{thm:konig}
 If $G$ is a bipartite graph with no isolated vertices, then $\overline{\beta}(G) = \alpha(G)$.
\end{theorem}

A {\it path cover} of a graph $G$ is a set of disjoint induced paths of $G$ that contains each vertex of $G$ and the {\it path cover number} of $G$, denoted $P(G)$, is the number of paths in the smallest path cover of $G$.  
In \cite{pathcover}, it was shown that the path cover number and the zero forcing number are the same for trees. Since both of those graph parameters are additive over connected components that result extends to acyclic graphs.

\begin{theorem}[\cite{pathcover}] \label{thm:Z(T)=P(T)}
  If $T$ is an acyclic graph, then $\Z(T) = P(T)$.
\end{theorem}

A main result of this paper, Theorem \ref{thm:G-Jforest}, gives an upperbound for the zero forcing number of cubic graphs that is useful for validating Conjecture \ref{theconjecture} in many cases and providing a bound that is slightly weaker in most other cases.

\begin{theorem}\label{thm:G-Jforest} Let $G$ be a cubic graph. If there exists an $S \subseteq V(G)$, such that $G-S$ is acyclic with $c$ components, then \[\Z(G) \le \alpha(G) +\beta(G[S]) + c.\]\end{theorem}

\begin{proof} 
Assume there exists an $S \subseteq V(G)$, such that $F=G-S$ is acyclic with $c$ components. If $B$ is a zero forcing set of $F$, then $S \cup B$ is a zero forcing set of $G$. So $Z(G) \le |S| + Z(F)$. 

Let $H$ be the bipartite graph created by taking the graph induced by the degree $3$ vertices of $F$ and adding a neighbor to any isolated vertices. 
 Each vertex of $H$ that is not degree $3$ in $F$ must be in a component of $H$ that is isomorphic to $K_2$ and adjacent to a vertex of degree 3 in $F$. 
 Therefore, there exists a maximum independent set of $H$ that only contains degree $3$ vertices of $F$.
 Let $A$ be such an independent set of $H$ and $E_H$ be the smallest edge cover of $H$. 
By Theorem \ref{thm:konig}, $|A| = |E_H|$. 
For each edge of $E_H$ that is not in $F$, replace it with an edge in $F$ that is incident to the same degree $3$ vertex in $F$.
Since $E_H$ covers all of the degree $3$ vertices in $F$, $\Delta(F-E_H) \le 2$. 
Therefore, $F-E_H$ is a path cover of $F$, since $F$ is acyclic.  
Let $k$ be the number of components of the acyclic graph $F-E_H$, so $k = |E_H| + c$ and $P(F) \le k$. 
By Theorem \ref{thm:Z(T)=P(T)}, 
\[\Z(F) = P(F) \le k = |E_H| + c \le |A| + c.\]


The vertices of $S$ can be partitioned into two sets, $S_1$ and $S_2$, such that $S_1$ is an independent set of $G[S]$ and $S_2$ is a vertex cover of $G[S]$. Each vertex in $A$ has degree $3$ in $F$, so there are no edges between $A$ and $S$. Therefore, $A \cup S_1$ is an independent set of $G$ and $|A| + |S_1| \le \alpha(G)$. 
Hence,

\[Z(G) \le |S| + Z(F) \le |S_1| + |S_2| + |A| + c \le \alpha(G) +\beta(G[S]) + c.\]

\end{proof}


\begin{corollary}\label{thm:3alpha}
    If $G$ is a cubic graph with no components isomorphic to $K_4$, then

    \[\Z(G) \le 3\alpha(G) - \dfrac{n}{2}.\]
\end{corollary}

\begin{proof}
Let $A$ be an independent set of $G$ such that each component $G-A$ is a path. 
Such an $A$ exists by Lemma \ref{lem:allpaths} and $G-A$ is acyclic.
Since every component of $G-A$ is a path, the number of components in $G-A$ is equal to $|V(G-A)| - |E(G-A)|$. 
Further, $G-A$ has $n-\alpha(G)$ vertices and $\dfrac{3n}{2} - 3\alpha(G)$ edges. 
Therefore, $G-A$ has $n-\alpha(G) - \left(\dfrac{3n}{2} - 3\alpha(G)\right) =$ $2\alpha(G) - \dfrac{n}{2}$ components. 
Applying Theorem \ref{thm:G-Jforest} yields $Z(G) \le \alpha(G) + 2\alpha(G) - \dfrac{n}{2}$.
\end{proof}

Theorem \ref{thm:G-Jforest} provides a very useful method of bounding bounding the zero forcing number in terms of the independence number. 
In particular, if the set $S$ in Theorem \ref{thm:G-Jforest} is independent and $G-S$ is a tree, then $G$ satisfies Conjecture \ref{theconjecture} because $\beta(G[S])=0$. 
A slightly weaker bound can be attained if $S$ is near independent since $\beta(G[S])=1$. Partitioning a graph in this way is known as decycling the graph.
Given a graph $G$, a set $S\subseteq V(G)$ is a {\it decycling set} if $G-S$ is acyclic.  A {\it decycling partition} of a connected graph $G$ is a partition $\{R,S\}$ of $V(G)$ such that $S$ is a decycling set for $G$ and $R = V(G)\backslash S$.
The {\it decycling number} of a graph $G$, denoted $\phi(G)$, is the minimum number of vertices in a decycling set.

In \cite{decycling} a strong connection between upper-embeddability and decycling sets in cubic graphs is shown.  A {\it cellular embedding} of a graph $G$ on a closed, orientable surface $\mathcal{S}$ is a drawing of $G$ on $\mathcal{S}$ where no edges cross and each face is homeomorphic to a disc. 
The maximum genus of a graph $G$, denoted by $\gamma_M(G)$, is the largest genus of a closed, orientable surface onto which $G$ has a cellular embedding.  
Graphs for which $\gamma_M(G) = \left\lfloor \dfrac{|E(G)| - |V(G)| + 1}{2}\right\rfloor$ are called {\it upper-embeddable}.  
Upper-embeddable graphs can be partitioned into graphs that have a cellular embedding on an orientable surface with one or two faces, respectively called {\it orientably one-face-embeddable} and {\it orientably two-face-embeddable}.  
In \cite{decyclegenus} $\phi(G) + \gamma_M(G) = n/2+1$ is proved for all cubic graphs.
This result has led to a characterization of cubic graphs based on independent and near independent decycling sets and embeddability.
A more in-depth discussion of decycling and embeddability can be found in \cite{decycling}.




\begin{theorem}[\cite{decycling}]\label{thm:decycle4.2}
For a connected cubic graph $G$, the following results are equivalent.
\begin{enumerate}
    \item[(i)] $G$ is orientably one-face-embeddable.
    
    \item[(ii)] The vertex set of $G$ can be partitioned into two sets $R$ and $S$ such that $R$ induces a tree and $S$ is independent.
\end{enumerate}

\end{theorem}

\begin{theorem}[\cite{decycling}] \label{thm:decycle4.3}
For a connected cubic graph $G$, the following results are equivalent.
\begin{enumerate}
    \item[(i)] $G$ is orientably two-face-embeddable.
    \item[(ii)] The vertex set of $G$ can be partitioned into two sets $R$ and $S$ such that either
        \begin{enumerate}
            \item[1.] $R$ induces a tree and $S$ is near independent, or
            \item[2.] $R$ induces a forest with two components and $S$ is  independent.
        \end{enumerate}
\end{enumerate}

\end{theorem}

Theorems \ref{thm:decycle4.2} and \ref{thm:decycle4.3} provide decycling partitions of cubic graphs that allow Theorem \ref{thm:G-Jforest} to be applied, giving the following result.

\begin{corollary}\label{cor:onefacetwoface}
    Let $G$ be a connected cubic graph.
    
        \begin{enumerate}
            \item  If $G$ is orientably one-face-embeddable, then $\Z(G) \le \alpha(G) + 1$.
            \item   If $G$ is orientably two-face-embeddable, then $\Z(G) \le \alpha(G) + 2$.
    
        \end{enumerate}
\end{corollary}

Corollary \ref{cor:onefacetwoface} states that if $G$ is an upper-embeddable connected cubic graph, then $\Z(G) \le \alpha(G) + 2$. 
Theorems \ref{thm:decycle1} and \ref{thm:decycle2} establish that almost every cubic graph is upper-embeddable. 
 Let $\mathbb{P}\{E\}$ denote the probability of event $E$ occurring. 
 For a sequence of probability spaces $\Omega_n, n\ge 1$, an event $E_n$ of $\Omega_n$ occurs asymptotically almost surely, shortened to a.a.s.,  if $\displaystyle\lim_{n\to\infty} \mathbb{P}\{E_n\} = 1$.

    \begin{theorem}[\cite{decyclerandom}]\label{thm:decycle1}
        A random cubic graph $G$ on $n$ vertices, a.a.s.\ satisfies

        \[\phi(G) = \left\lceil \dfrac{n + 2}{4}\right\rceil.  \]

    \end{theorem}

\begin{theorem}[\cite{decyclinggenus}]\label{thm:decycle2}

    A connected cubic graph $G$ is upper-embeddable if and only if \[\phi(G) = \left\lceil\dfrac{n + 2}{4}\right\rceil.\]

\end{theorem}

\begin{theorem}\label{thm:almostall} 
A connected cubic graph $G$ a.a.s.\ satisfies $\Z(G) \le \alpha(G) + 2.$

\end{theorem}

\begin{proof}  Theorem \ref{thm:decycle1} and \ref{thm:decycle2}  imply that a connected cubic graph $G$ a.a.s.\ is upper-emdebbable, which means, by Corollary \ref{cor:onefacetwoface}, $G$ also a.a.s.\ satisfies $\Z(G) \le \alpha(G) + 2$.\end{proof}

\subsection*{Infinite Family where $\Z(G) = \alpha(G) + 1$}

An infinite family of cubic graphs with $\Z(G) = \alpha(G) + 1$ will be established to show that Conjecture \ref{theconjecture} would provide a tight bound.
A {\it $3$-$1$ tree} is a tree in which every vertex has a degree of either $1$ or $3$.  In a tree a degree $1$ vertex is called a {\it leaf}.

\begin{lemma}\label{lem:noleafchain}
   
    If $T$ is a $3$-$1$ tree and $5\le |V(T)|$, then there exists a minimum zero forcing set of $T$ that has a set of forces such that every leaf of the zero forcing set performs a force.
\end{lemma}

\begin{proof}Let $B$ be a minimum zero forcing set of $T$ with a set of forces that has the fewest number of leaves that do not perform a force.  
Assume there is a leaf $\ell \in B$ that does not perform a force.
    
    Let $v\in V(T)$ be adjacent to $\ell$ which implies $\deg(v) = 3$ since $T$ is not $K_2$. 
    Let $N(v) = \{v_1,v_2,\ell\}$.
    Since $\ell\in B$, it follows that $v\notin B$. 
    Without loss of generality, assume that $v_1$ forces $v$.  
    This means that $v$ forces $v_2$, otherwise $v$ could force $\ell$ and $B\backslash \{\ell\}$ would be a smaller zero forcing set.

    If $\deg(v_1) = 3$, then change the forcing process so that $\ell$ forces $v$ rather than $v_1$.
    All other forces will remain the same in coloring the entire graph blue.
    Therefore, the number of leaves that do not perform a force has been reduced, a contradiction.

    If $\deg(v_1) = 1$, then $(B\backslash\{\ell\})\cup \{v_2\}$ is a zero forcing set where $v_1$ forces $v$ and $v$ forces $\ell$. $T$ is not $K_{1,3}$, so $v_2$ is not a leaf. Therefore, the number of leaves that do not perform a force has been reduced, a contradiction.\end{proof}

Let $G$ be a graph, $v \in V(G)$, and $H$ be a graph with $deg_G(v)$ vertices of degree $2$ and $|V(H)| - deg_G(v)$ vertices of degree $3$.  The graph $H$ {\it replaces} $v$ in $G$ by deleting $v$ and adding a matching between the neighbors of $v$ and the degree $2$ vertices of $H$. 
 In this paper vertices are replaced by graphs where the resulting graph is unique, even though this is not necessarily the case for all replacements. 
Define $G_T$ to be the graph obtained from tree $T$ by replacing every leaf of $T$ with $K_4$ with a subdivided edge.

 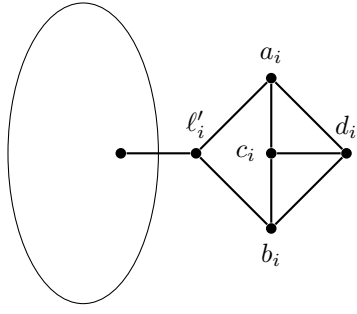
\begin{figure}[h!]
    \centering
    \captionsetup{width=0.8\textwidth}

    \begin{tikzpicture}

        \draw (.25,0) ellipse (1cm and 2cm);

        \node[circle,fill=black,inner sep=0pt,minimum size=4pt,label=above:{$\ell_i'$}] (ell_i2) at (1.75,0) {};
        \node[circle,fill=black,inner sep=0pt,minimum size=4pt,label=above:{}] (v2) at (0.75,0) {};

        \node[circle,fill=black,inner sep=0pt,minimum size=4pt,label=below:{$b_i$}] (b_i2) at (2.75,-1) {};
        \node[circle,fill=black,inner sep=0pt,minimum size=4pt,label=above:{$d_i$}] (d_i2) at (3.75,0) {};

        \node[circle,fill=black,inner sep=0pt,minimum size=4pt,label=above:{$a_i$}] (a_i2) at (2.75,1) {};

        \node[circle,fill=black,inner sep=0pt,minimum size=4pt,label= left:{$c_i$}] (c_i2) at (2.75,0) {};

        \Edge(v2)(ell_i2)
	\Edge(ell_i2)(b_i2)
	\Edge(b_i2)(d_i2)
        \Edge(a_i2)(d_i2)
        \Edge(c_i2)(d_i2)
        \Edge(a_i2)(ell_i2)
        \Edge(c_i2)(a_i2)
        \Edge(c_i2)(b_i2)

	 \end{tikzpicture}
    \caption{A leaf of $T$ replaced by a $K_4$ with a subdivided edge.}\label{fig:leaftodiamond}
\end{figure}

\begin{theorem}\label{thm:3-1tree}
    If $T$ is a $3$-$1$ tree on $n$ vertices, then 
    

    \[\Z(G_T) = \Z(T) + n+2 = \alpha(G_T) + 1.\]
\end{theorem}

\begin{proof}
Assume $T$ is a $3$-$1$ tree with $L$ leaves.
Let $K_4$ with a subdivided edge replace leaf $\ell_i$ of $T$, for $1 \le i \le L$, and be labelled by vertices $a_i,b_i,c_i,d_i$, and $\ell_i'$ in $G_T$ as in Figure \ref{fig:leaftodiamond}.
If $T= K_{1,3}$, then $\Z(G_T) = 8$ and $\alpha(G_T)=7$. This satisfies the statement, so assume that $T$ is not $K_{1,3}$. Let $T_{s}$ to be the subgraph of $T$ induced by the degree 3 vertices of $T$.   
By Lemma \ref{lem:noleafchain}, there exists a minimum zero forcing set of $T$, $B$, with a set of forces such that every leaf in $B$ performs a force. Let $V_H$ be the set of forcing chains that correspond to this set of forces.
Create an auxiliary graph $H$ with vertex set $V_H$ and two vertices adjacent in $H$ if there is an edge between the two forcing chains. 
The graph $H$ is a tree with $Z(T)$ vertices and since each forcing chain contains a vertex of degree $3$, every edge of $H$ is also an edge in $T_S$.
Therefore, $E(H)$ corresponds to an edge cover of $T_s$ which implies $\overline{\beta}(T_s) \le |E(H)| = \Z(T_s) - 1$.

Let $P$ be a minimum edge cover of $T_s$ and consider the graph $P' =T - P$.
There are $\overline{\beta}(T_s) + 1$ components of $P'$ and each is a path.
Since the path cover number and the zero forcing number are the same for trees, by Theorem \ref{thm:Z(T)=P(T)}, it follows that $\Z(T) \le \overline{\beta}(T_s) + 1$ and thus $\overline{\beta}(T_s)=\Z(T)-1$.
Since $T$ is not $K_{1,3}$,  $T_s$ is a non-trivial tree so Theorem \ref{thm:konig} implies that $\alpha(T_s) = \Z(T) - 1$.

Let $A$ be a maximum independent set of $T_s$ and $A' = A\cup \{a_i,b_i \, | \, 1\le i \le L\}$. 
$A'$ is independent in $G_T$, so $\alpha(T_s) + 2L \le \alpha(G_T)$. Therefore, $(\Z(T) - 1) + 2L \le \alpha(G_T).$

Conversely, assume that $A'$ is a maximum independent set of $G_T$ such that $\alpha(T_s) + 2L < |A'|$.
$A'$ has at most two vertices from $\{a_i,b_i,c_i,d_i,\ell_i'\}$.
This implies $A'$ has more than $\alpha(T_s)$ vertices from $T_s$, a contradiction.
Thus, $\alpha(G_T) \le \alpha(T_s) + 2 L$ which means

\begin{equation}\label{3-1eq3}
\Z(T) + 2L = \alpha(G_T) +1.
\end{equation}

Recall $B$ is a minimum zero forcing set of $T$ and let $B' = (B\backslash\{\ell_i\, |\, 1\le i \le L\})\cup \{ a_i,c_i\, | \, 1\le i \le L \} \cup \{d_i \, | \,  \ell_i \in B\}$.
Since $B'$ is a zero forcing set for $G_T$, $\Z(G_T) \le \Z(T) + 2L$.

Assume that $\Z(G_T)<\Z(T)+2L$ and let $B'$ be a minimum zero forcing set of $G_T$. 
Since $\{a_i,b_i\}$ and $\{c_i,d_i\}$ are disjoint forts in $G_T$, $2 \le |B' \cap \{a_i,b_i,c_i,d_i\}|$, which implies that $|B' \cap V(T)| < \Z(T)$. 
Without loss of generality, let $a_i,c_i\in B'$ for $1\le i \le L$. 

If $|B' \cap \{a_i,b_i,c_i,d_i\}| = 4$ for some $i$, then $B'\backslash\{b_i\}$ is a zero forcing set for $G_T$, which contradicts that $B'$ is a minimum zero forcing set for $G_T$.
If $|B' \cap \{a_i,b_i,c_i,d_i\}| = 3$ for any $i$, then a new zero forcing set for $G_T$ can be constructed by removing $b_i$ (or $d_i$) from $B'$ and replacing it with $\ell_i'$.
If $|B' \cap \{a_i,b_i,c_i,d_i\}| = 2$ for all $1\le i \le L$, then for each $i$, $b_i$ and $d_i$ cannot be forced until $\ell_i'$ is blue.  
Therefore, $B'\cap V(T)$ must be a zero forcing set of $T$, which contradicts $|B' \cap V(T)| < \Z(T)$.  
Thus,

\begin{equation}\label{3-1eq2}\Z(G_T) = Z(T) + 2 L.\end{equation}

Since $T$ is a $3$-$1$ tree $L = \frac{n+2}{2}$ and equations \ref{3-1eq3} and \ref{3-1eq2} give the desired result.\end{proof}

Within the proof of Theorem \ref{thm:3-1tree} it is also shown that $\Z(G_T) = \Z(T) + 2L = \alpha(G_T) + 1$ where $L$ is the number of leaves in $3$-$1$ tree $T$.


\section{Bounds on Subcubic Graphs}\label{sec:subcubic}

This section uses previously established bounds on cubic graphs to get bounds on subcubic graphs. 
This is accomplished by starting with a subcubic graph and using replacements to change it into a cubic graph while controlling the zero forcing number and independence number.  

Let $H_1$ be the graph obtained by replacing a degree 2 vertex of $K_4-e$ with $K_4-e$.
Define $G_1(v)$ to be the graph obtained from $G$ by replacing a degree $1$ vertex $v$ with $H_1$ (see Figure \ref{fig:DegreeOneBlowup}).  
The notation for $G_1(v)$ will be shortened to $G_1$ when it is clear which vertex is being replaced.
This replacement will be used to alter the degree 1 vertices of a graph while controlling the zero forcing number and independence number.

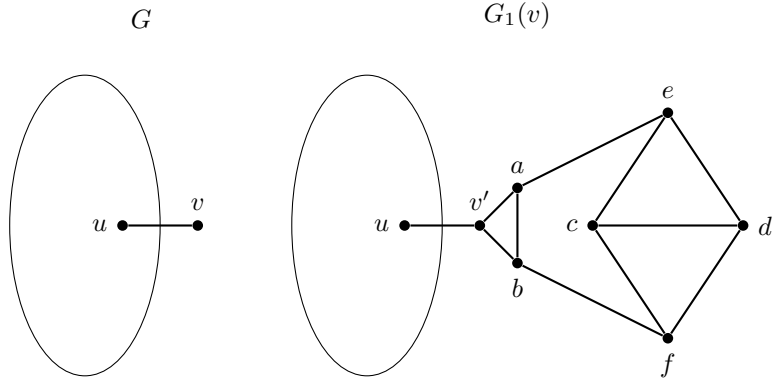
\begin{figure}[h!]
    \centering
    \captionsetup{width=0.8\textwidth}
    \begin{tikzpicture}

        Graph of G
        \draw (-3.5,0) ellipse (1cm and 2cm) (-2.75,2.5)  node [text=black,above] {$G$};
        \node[circle,fill=black,inner sep=0pt,minimum size=4pt,label=above:{$v$}] (y) at (-2,0) {};
        \node[circle,fill=black,inner sep=0pt,minimum size=4pt,label=left:{$u$}] (x) at (-3,0) {};

        \draw (.25,0) ellipse (1cm and 2cm) (2.25,2.5)  node [text=black,above] { $G_1(v)$};
        \node[circle,fill=black,inner sep=0pt,minimum size=4pt,label=above:{$v'$}] (y1) at (1.75,0) {};
        \node[circle,fill=black,inner sep=0pt,minimum size=4pt,label=left:{$u$}] (x1) at (0.75,0) {};

        \node[circle,fill=black,inner sep=0pt,minimum size=4pt,label=above:{$a$}] (a) at (2.25,0.5) {};

        \node[circle,fill=black,inner sep=0pt,minimum size=4pt,label=below:{$b$}] (b) at (2.25,-0.5) {};

        \node[circle,fill=black,inner sep=0pt,minimum size=4pt,label=below:{$f$}] (f) at (4.25,-1.5) {};

        \node[circle,fill=black,inner sep=0pt,minimum size=4pt,label=above:{$e$}] (e) at (4.25,1.5) {};

        \node[circle,fill=black,inner sep=0pt,minimum size=4pt,label=left:{$c$}] (c) at (3.25,0) {};

        \node[circle,fill=black,inner sep=0pt,minimum size=4pt,label=right:{$d$}] (d) at (5.25,0) {};

	\Edge(x)(y)
        \Edge(x1)(y1)
        \Edge(y1)(a)
        \Edge(y1)(b)
        \Edge(a)(b)
        \Edge(a)(e)
        \Edge(b)(f)
        \Edge(c)(e)
        \Edge(c)(d)
        \Edge(f)(d)
        \Edge(c)(f)
        \Edge(e)(d)

	 \end{tikzpicture}
    \caption{Graphs $G$ and $G_1(v)$.}\label{fig:DegreeOneBlowup}
\end{figure}

\begin{lemma}\label{lem:subcubicleaves} If $G$ is a graph, $v\in V(G)$, and $\deg(v) = 1$, then $\alpha(G_1(v)) = \alpha(G) + 2$ and $\Z(G_1(v)) = \Z(G) + 2$.
\end{lemma}

\begin{proof}

 Let $H_1$ be the graph that replaces $v$ in $G$, and be labelled by vertices $a,b,c,d,e,f$, and $v'$ in $G_1(v)$ as in Figure \ref{fig:DegreeOneBlowup}.
Let $A$ be a maximum independent set of $G$. 
Since $A \cup \{e,f\}$ is independent in $G_1$, it follows that $\alpha(G) + 2 \le \alpha(G_1)$. 
Also, $\alpha(H_1-v') = 2$, so $\alpha(G_1) \leq \alpha(G)+2$. 
Therefore, $\alpha(G_1) =\alpha(G)+2$.

Let $B$ be a minimum zero forcing set of $G$ that does not contain $v$ and $uv\in E(G)$. 
This means that $u$ forces $v$ in $G$. 
If $B'= B \cup \{a,c\}$, then $B'$ is a zero forcing set of $G_1$ since $u$ can force $v'$ in $G_1$. 
So $\Z(G_1) \le \Z(G) + 2$.  

Let $B'$ be a minimum zero forcing set for $G_1$. 
Since $\{a,b,c,e\}$, $\{a,b,c,f\}$, $\{a,b,d,e\}$, $\{a,b,d,f\}$, $\{a,b,e,f\}$, and $\{c,d\}$ are forts in $G_1$, $2 \le |B' \cap \{a,b,c,d,e,f\}|$. 
If $|B' \cap \{a,b,c,d,e,f\}| = 2$, then, without loss of generality, $a,c \in B'$; which implies $B' \cap V(G)$ is a zero forcing set of $G$. 
So $\Z(G) \le \Z(G_1) - 2$. 
If $3 \le |B' \cap \{a,b,c,d,e,f\}|$, then $\Z(G) \le \Z(G_1) - 2$ because $(B'\cap V(G)) \cup \{v\}$ is a zero forcing set of $G$. 
Therefore, $\Z(G_1) = \Z(G) + 2$.\end{proof}

Define $G_2(v)$ to be the graph obtained from $G$ by replacing a degree $2$ vertex $v$ in $G$ with $K_4 - e$. The notation for $G_2(v)$ will be shortened to $G_2$, when it is clear which vertex is being replaced.

 \begin{figure}[h!]
     \centering
     \captionsetup{width=0.8\textwidth}
     
    \begin{tikzpicture}
     \draw (-3,0) ellipse (1cm and 2cm) (-1.75,2.5) node [text=black,above] {$G$};
   
     \node[circle,fill=black,inner sep=0pt,minimum size=4pt,label=left:{$u_1$}] (x) at (-2.5,0.5) {};
     \node[circle,fill=black,inner sep=0pt,minimum size=4pt,label=right:{$v$}] (y) at (-0.5,0) {};
     \node[circle,fill=black,inner sep=0pt,minimum size=4pt,label=left:{$u_2$}] (z) at (-2.5,-0.5) {};
     Edges in G
    \Edge (x)(y);
     \Edge (y)(z);

     \draw (2,0) ellipse (1cm and 2cm) (3.25,2.5) node [text=black,above] {$G_2(v)$};
    
     \node[circle,fill=black,inner sep=0pt,minimum size=4pt,label=above:{$v_1$}] (c) at (4.5,1) {};
    
     \node[circle,fill=black,inner sep=0pt,minimum size=4pt,label=left:{$u_1$}] (x1) at (2.5,0.5) {};

     \node[circle,fill=black,inner sep=0pt,minimum size=4pt,label=left:{$u_2$}] (z1) at (2.5,-0.5) {};

     \node[circle,fill=black,inner sep=0pt,minimum size=4pt,label=below:{$v_2$}] (d) at (4.5,-1) {};

     \node[circle,fill=black,inner sep=0pt,minimum size=4pt,label=left:{$a$}] (a) at (4,0) {};

     \node[circle,fill=black,inner sep=0pt,minimum size=4pt,label=right:{$b$}] (b) at (5,0) {};

    \Edge(x1)(c)
    \Edge(a)(d)
    \Edge(c)(a)
    \Edge(c)(b)
    \Edge(b)(a)
    \Edge(d)(b)
    \Edge(d)(z1)

 \end{tikzpicture}
 \caption{Graphs $G$ and $G_2(v)$.}\label{fig:DegreeTwoBlowup}
 \end{figure}
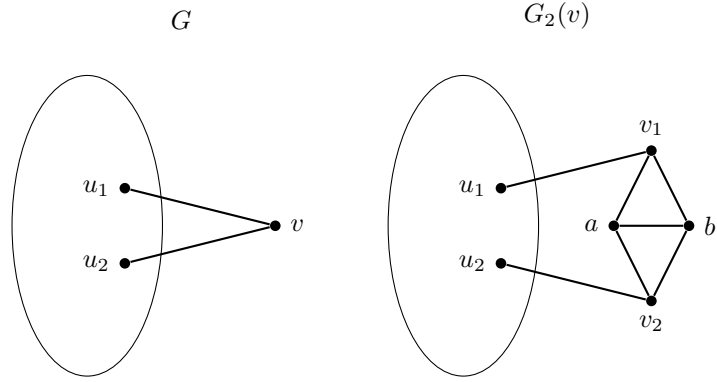

\begin{lemma}\label{lem:subcubicdegtwo}
If $G$ is a graph, $v\in V(G)$, and $\deg(v) = 2$, then $\alpha(G_2(v)) = \alpha(G) + 1$ and $\Z(G_2(v)) = \Z(G) + 1$.
\end{lemma}
\begin{proof}Let $K_4-e$ replace $v$ in $G$, and be labelled by vertices $a,b,v_1$, and $v_2$ in $G_2(v)$ as in Figure \ref{fig:DegreeTwoBlowup}.
Let $A$ be a maximum independent set of $G$. 
    If $v\in A$, then $A' = (A\backslash\{v\})\cup \{v_1,v_2\}$ is independent in $G_2$. 
    If $v\notin A$, then $A\cup\{a\}$ is independent in $G_2$. 
    So $\alpha(G) + 1 \le \alpha(G_2)$.

    Let $A'$ be a maximum independent set of $G_2$. 
    Since $A' \setminus \{v_1,v_2,a,b\}$ is independent in $G$, it follows that $\alpha(G_2) \le \alpha(G) + 2$. 
    If $\alpha(G_2) = \alpha(G) + 2$, then $v_1,v_2 \in A'$; which implies that $(A'\backslash \{v_1,v_2\}) \cup \{v\}$ is independent in $G$ and has more than $\alpha(G)$ vertices, a contradiction. 
    Therefore, $\alpha(G_2) = \alpha(G) + 1$.

     Let $B$ be a minimum zero forcing set of $G$ that does not contain $v$, then $B \cup \{a\}$ is a zero forcing set for $G_2$. 
     So $\Z(G_2) \le \Z(G) + 1$.

    Let $B'$ be a minimum zero forcing set for $G_2$ that does not have vertex $a$ and $u_1v,u_2v\in E(G)$.  
    Since $\{a,b\}$ is a fort in $G_2$, $b$ must be in $B'$.  
    If $u_1$ does not force $v_1$ and $u_2$ does not force $v_2$, then $v_1$ or $v_2$ must be in $B'$.  
    If both $v_1$ and $v_2$ are in $B'$, then $(B' \cup \{u_1,v\}) \setminus \{v_1,v_2,b\}$ is a zero forcing set of $G$ implying $Z(G) \le Z(G_2) - 1$.  
    However, if only one of $v_1$ or $v_2$ is in $B'$, then $(B' \cap V(G))\cup \{v\}$ is a zero forcing set of $G$ also implying $Z(G) \le Z(G_2) - 1$.

     If $u_1$ forces $v_1$ or $u_2$ forces $v_2$, then $(B' \cap V(G)) \setminus \{b\}$ is a zero forcing set of $G$; which implies $Z(G) \le Z(G_2) - 1$.  
     Therefore, $\Z(G_2) = \Z(G) + 1$.
     \end{proof}

\begin{theorem}\label{thm:cubictosubcubic}
If $c\in \mathbb{Z}$ and $\Z(H) \leq \alpha(H)+c$ for all connected cubic graphs $H$, then $\Z(G) \leq \alpha(G)+c$ for all connected subcubic graphs $G$. 
\end{theorem}

 \begin{proof} Let $G$ be a connected subcubic graph and $H$ be obtained from $G$ by replacing every degree $1$ vertex in $G$ with $H_1$ and every degree $2$ vertex of $G$ with $K_4-e$. 
 By Lemmas \ref{lem:subcubicleaves} and \ref{lem:subcubicdegtwo}, in each instance that a vertex is replaced the zero forcing number and the independence number both change by the same amount.  
This implies that $\alpha(H) - \Z(H) = \alpha(G) - \Z(G)$.  
Therefore, $\Z(G) \le \alpha(G) + c$ since $\Z(H) \le \alpha(H) + c$.\end{proof}

In \cite{DH} it was shown that $\Z(G) \le \alpha(G) + 1$ for connected, claw-free, cubic graphs. 
The proof of Theorem \ref{thm:cubictosubcubic} uses replacements to change subcubic graphs into cubic graphs while the zero forcing number and the independence number change by the same amount. 
Since the graphs used to replace vertices, $H_1$ and $K_4-e$, are both claw free, this technique can be applied to claw-free subcubic graphs to give the following result.   

\begin{theorem}\label{thm:subcubclawfreealpha+1}
    If $G\neq K_4$ is a connected, claw-free, subcubic graph, then \[\Z(G) \le \alpha(G) + 1.\]
\end{theorem}


However, this can be generalized in a way that does not require the graph to be to be claw free by replacing claw centers. Theorem \ref{thm:subcubicbound} give such a result.

Define $G'(v)$ to be the graph obtained from replacing degree $3$ vertex $v$ in $G$ with $K_3$. The notation for $G'(v)$ will be shortened to $G'$, when it is clear which vertex is being replaced. 

\begin{figure}[h!]
    \centering
    \captionsetup{width=0.8\textwidth}
   
    \begin{tikzpicture}
        \draw (-1,0) ellipse (2cm and 1cm) (-1,1)  node [text=black,above] {$G$};
        
        \node[circle,fill=black,inner sep=0pt,minimum size=4pt,label=above:{$u_1$}] (u_1) at (-2,-.5) {};
        
        \node[circle,fill=black,inner sep=0pt,minimum size=4pt,label=above:{$u_2$}] (u_2) at (-1,-.5) {};

        \node[circle,fill=black,inner sep=0pt,minimum size=4pt,label=above:{$u_3$}] (u_3) at (0,-.5) {};

        \node[circle,fill=black,inner sep=0pt,minimum size=4pt,label=below:{$v$}] (c) at (-1,-2) {};

        \draw (4,0) ellipse (2cm and 1cm) (4,1)  node [text=black,above] {$G'(v)$};
        
        \node[circle,fill=black,inner sep=0pt,minimum size=4pt,label=above:{$u_1$}] (u_11) at (3,-.5) {};
        
        \node[circle,fill=black,inner sep=0pt,minimum size=4pt,label=above:{$u_2$}] (u_21) at (4,-.5) {};

        \node[circle,fill=black,inner sep=0pt,minimum size=4pt,label=above:{$u_3$}] (u_31) at (5,-.5) {};

        \node[circle,fill=black,inner sep=0pt,minimum size=4pt,label=below:{$v_2$}] (v2) at (4,-2) {};

        \node[circle,fill=black,inner sep=0pt,minimum size=4pt,label=below:{$v_1$}] (v1) at (3,-3) {};

        \node[circle,fill=black,inner sep=0pt,minimum size=4pt,label=below:{$v_3$}] (v3) at (5,-3) {};
		
        \Edge(c)(u_1)
        \Edge(c)(u_2)
        \Edge(c)(u_3)
        \Edge(v1)(u_11)
        \Edge(v2)(u_21)
        \Edge(v3)(u_31)
        \Edge(v1)(v2)
        \Edge(v2)(v3)
        \Edge(v3)(v1)
	 	
	 \end{tikzpicture}
    \caption{Graphs $G$ and $G'(v)$.}
    \label{fig:triangle}
\end{figure}
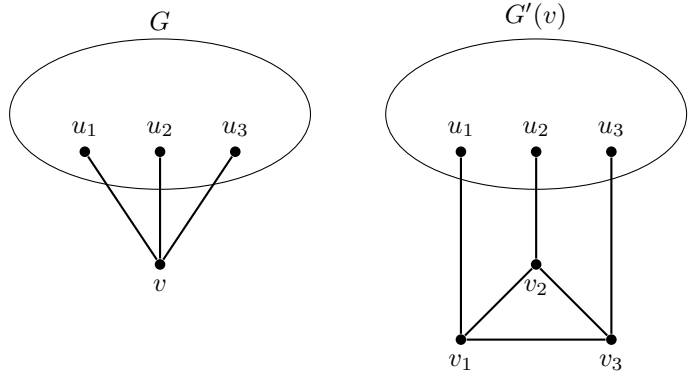

\begin{lemma}\label{lem:subcubictriangle}
    If $G$ is a graph, $v\in V(G)$, and $\deg(v) = 3$, then $\alpha(G'(v)) \le \alpha(G) + 1$ and $\Z(G) \le \Z(G'(v))$.
\end{lemma}

\begin{proof}

Let $K_3$ replace $v$ in $G$, and be labelled by vertices $v_1,v_2$, and $v_3$ in $G'(v)$ as in Figure \ref{fig:triangle}.
Let $A'$ be a maximum independent set of $G'$. 
There is at most one vertex of $A'$ in $\{v_1,v_2,v_3\}$ and $A' \setminus \{v_1,v_2,v_3\}$ is an independent set of $G$. 
Therefore, $\alpha(G') \leq \alpha(G) + 1$.

Let $B'$ be a minimum zero forcing set of $G'$ that does not contain $v_3$.  Also, let $u_1,u_2$, and $u_3$ be the neighbors of $v$ in $G$ and the neighbors of $v_1, v_2$, and $v_3$, respectively, in $G'$.

Assume $v_1,v_2\in B'$ and, without loss of generality, $v_2$ forces $v_3$ or $u_3$ forces $v_3$ in $G'$. 
 Let $B = (B' \setminus \{v_1,v_2\}) \cup \{v,u_1\}$.

If $v_2$ forces $v_3$ in $G'$, every force in $G'$ before $v_2$ forces $v_3$ can be performed in $G$, resulting in $u_1$ and $u_2$ being blue. This allows $v$ to force $u_3$ and the process can continue to force all of $G$, so $B$ is a zero forcing set of $G$.

If $u_3$ forces $v_3$ in $G'$, every force in $G'$ before $u_3$ forces $v_3$ can be performed in $G$, resulting in $u_1$ and $u_3$ being blue. This allows $v$ to force $u_2$ and the process can continue to force all of $G$, so $B$ is a zero forcing set of $G$.

Assume $v_1\in B'$ and $v_2\notin B'$.  In $G'$, $v_1$ cannot perform a force until either $u_2$ forces $v_2$ or $u_3$ forces $v_3$. Without loss of generality, assume that $u_2$ forces $v_2$.

If $u_1$ must force some vertex $x \neq v_1$ before $u_2$ forces $v_2$, then let $B = (B' \setminus \{v_1\}) \cup \{x\}$. Every force in $G'$ before $u_1$ forces $x$ can be performed in $G$ and every force in $G'$ that occurs after $u_1$ forces $x$ but before $u_2$ forces $v_2$ can also be performed in $G$. This allows $u_2$ to force $v$ and the process can continue to force all of $G$, so $B$ is a zero forcing set of $G$.

If $u_1$ does not need to perform a force before $u_2$ forces $v_2$, then let $B = (B' \setminus \{v_1\}) \cup \{u_1\}$. Every force in $G'$ before $u_2$ forces $v_2$ can be performed in $G$. This allows $u_2$ to force $v$ and the process can continue to force all of $G$, so $B$ is a zero forcing set of $G$.

Assume $\{v_1,v_2,v_3\}\cap B' = \emptyset$ and, without loss of generality, that $u_2$ forces $v_2$ before $u_3$ forces $v_3$. Every force in $G'$ before $u_3$ forces $v_3$ can be performed in $G$, with $u_2$ forcing $v$ and resulting in $u_2$ and $u_3$ being blue. 
This allows $v$ to force $u_1$ and the process can continue to force all of $G$, so $B'$ is also a zero forcing set of $G$.\end{proof}

\begin{theorem}\label{thm:subcubicbound}If $G\neq K_4$ is a connected subcubic graph with $C$ claw centers, then

\[\Z(G) \le \alpha(G) + 1 +C.\]
\end{theorem}
\begin{proof}
    Let $H$ be obtained from $G$ by replacing every claw center of $G$ with $K_3$.
    By Lemma \ref{lem:subcubictriangle}, $\Z(G) \le \Z(H)$ and $\alpha(H) \le \alpha(G) + C$.  Since $H$ is claw-free, Theorem \ref{thm:subcubclawfreealpha+1} implies $\Z(H) \le \alpha(H) + 1$. 
\end{proof}

Lemma \ref{lem:upper-odd-deg} allows Corollary \ref{cor:onefacetwoface} to be generalized to subcubic graphs as it implies that $G_1(v)$ and $G_2(v)$ preserve upper-embeddability when $\deg(v) < 3$.

\begin{lemma}[\cite{upperembedJPX}]\label{lem:upper-odd-deg}
If $G$ is a graph, $v \in V(G)$ has odd degree, and $G-v$ is upper-embeddable, then $G$ is upper-embeddable. 
\end{lemma}

Similar to the proof of
Theorem \ref{thm:cubictosubcubic}, the proof of  Theorem \ref{thm:subcubicupper} uses replacements to control the zero forcing number, the independence number, and, by Lemma \ref{lem:upper-odd-deg}, upper-embeddability.

\begin{theorem}\label{thm:subcubicupper}
  If $G$ is a connected, subcubic, upper-embeddable graph, then $Z(G) \le \alpha(G) + 2$.
\end{theorem}

\begin{proof}
If $G$ is cubic the result holds, so assume there exists a vertex $v \in V(G)$ such that $\deg(v)<3$.

If $\deg(v)=1$, then let $H_1$ be the graph that replaces $v$ in $G$, and be labelled by vertices $a,b,c,d,e,f$, and $v'$ in $G_1(v)$ as in Figure \ref{fig:DegreeOneBlowup}.
Delete the vertices of $G_1$ in the following order to attain a graph that is isomorphic to $G$: $d,b,f,c,e$, and $a$.
At each step the vertex that is deleted has odd degree so Lemma \ref{lem:upper-odd-deg} can be applied.
Therefore, $G_1$ is also upper-embeddable.

If $\deg(v)=2$, then let $K_4-e$ replace $v$ in $G$, and be labelled by vertices $a,b,v_1$, and $v_2$ in $G_2(v)$ as in Figure \ref{fig:DegreeTwoBlowup}.
Subdividing both of the edges incident to $v$ in $G$ is isomorphic to $G_2-b$, which is upper-embeddable.  
By Lemma \ref{lem:upper-odd-deg}, $G_2$ is upper-embeddable since $G_2-b$ is upper-embeddable. 

Therefore, $G_1$ and $G_2$ are upper-embeddable, connected, subcubic graphs that have fewer vertices of degree less than 3 than $G$ has.  
Recursively replacing vertices with degree at most $2$ yields an upper-embeddable, connected cubic graph $H$ such that $\alpha(H)-\Z(H)=\alpha(G)-\Z(G)$.  
Therefore, $\Z(G) \leq \alpha(G)+2$ by Corollary \ref{cor:onefacetwoface}.
\end{proof}


\section{Conclusion}\label{sec:conc}
This paper primarily focuses on cubic and subcubic graphs. 
However, some of the ideas and methods can be generalized to get results and bounds relating zero forcing and independence numbers for other graphs.  
For example, by using chromatic number bounds, it can be shown that if the zero forcing number is small enough, then the zero forcing number is bounded by the independence number. 
The {\it chromatic number} of graph $G$, denoted $\chi(G)$, is the minimum number of colors needed to assign colors to the vertices of $G$ such that adjacent vertices have different colors. 
The following are well known results about the chromatic number of a graph.

\begin{theorem}[\cite{WW}] \label{thm:chialpha}
 For any graph $G$ on $n$ vertices, $\displaystyle \frac{n}{\chi(G)} \le \alpha(G)$.
\end{theorem}

\begin{theorem}[\cite{KAT}] \label{thm:chiz}
For any graph $G$, $\chi(G) \le \Z(G)+1$.
\end{theorem}

Combining Theorem \ref{thm:chialpha} and Corollary \ref{thm:chiz} gives Corollary \ref{cor:zalpha}.

\begin{corollary}\label{cor:zalpha}
For any graph $G$ on $n$ vertices, $\displaystyle \frac{n}{\Z(G) + 1} \le \alpha(G) $.
\end{corollary}

As an immediate consequence of Corollary \ref{cor:zalpha}, it follows that if $G$ is a connected graph on $n$ vertices and $\Z(G) \le \left\lceil\dfrac{n}{\Z(G) + 1}\right\rceil$, then $\Z(G) \le \alpha(G)$.  Corollary \ref{cor:Z<a2} makes use of this fact.

\begin{corollary}\label{cor:Z<a2}
    If $G$ is a graph on $n$ vertices and $\Z(G) \le \sqrt{n}$, then $\Z(G) \le \alpha(G)$.
\end{corollary}

\begin{proof}
Since $\Z(G) \leq \sqrt{n}$, 
\[\Z(G)=\left\lfloor\frac{[\Z(G)]^2-1}{\Z(G)+1}\right\rfloor+1 \leq \left\lfloor\frac{n-1}{\Z(G)+1}\right\rfloor+1=\left\lceil\frac{n}{\Z(G)+1}\right\rceil \leq \alpha(G).\]
\end{proof}

While Conjecture \ref{theconjecture} still remains open, this paper provides evidence supporting its validity by proving the bound for many more cubic graphs and providing a very close bound for most cubic graphs. 
In fact, there is also evidence to believe that the conjecture generalizes such that  for any connected graph $G$ that is not complete, 
\[\Z(G) \le (\Delta(G) - 2)\alpha(G) + (\Delta(G) -2).\] Theorem \ref{thm:Delta-1} is proved as many other results in this paper are: identify a maximum independent set then modify it to create a zero forcing set.

\begin{theorem}\label{thm:Delta-1}
            If $G$ is connected and not a complete graph with $3 \le \Delta(G)$, then \[\Z(G) \le (\Delta(G) -1)\alpha(G).\]
\end{theorem} 
\begin{proof}
    This proof is inductive on $\Delta$.
    Assume $\Delta(G) =3$. Let $A$ be a maximum independent set for $G$ and $H = G-A$.  
    Each connected component of $H$ is a path or cycle. If $a \in A$ with neighbors $x,y \in V(H)$ such that $\deg_H(x) = \deg_H(y) = 2$, then $x$ and $y$ must be adjacent; otherwise, $(A\backslash\{a\}) \cup \{x,y\}$ is a larger independent set. So each vertex of $A$ has neighbors in at most one cycle of $H$.
    
    Let $B$ be a set of vertices that contains an endpoint from each path of $H$ and a vertex from each cycle of $H$ that is also in a $K_3$, if such a vertex exists. Since $B$ contains at most one vertex from each component of $H$ it is an independent set and $|B| \le \alpha(G)$.
    
    Color the vertices of $G$ in $A \cup B$ blue. Each path of $H$ is able to be forced blue. Once all of the paths of $H$ are all blue, the cycles can be forced in the following way, since no vertex of $A$ is adjacent to any two distinct cycles. If a cycle of $H$ has a blue vertex, then its neighbor in $A$ can force their common neighbor that is also on the same cycle. Therefore, the entire cycle can be forced blue. If a cycle of $H$ does not have a blue vertex, then each vertex of the cycle can be forced by its neighbor in $A$ and the entire cycle is forced blue. Therefore, $A \cup B$ is a zero forcing set of $G$ and $|A \cup B| \le 2\alpha(G)$, so $\Z(G) \le 2\alpha(G)$.

    Assume that if $G$ is a graph that is not complete and $3\le \Delta(G) \le r$, then $\Z(G) \le (\Delta(G)-1)\alpha(G)$.
    Let $G$ be a graph with $\Delta(G) = r+1$ and $A$ be a maximum independent set of $G$. 
    Let the connected components of $G - A$ be $C_1, \ldots, C_k$. 
    
    If $C_i$ is not isomorphic to $K_{r+1}$, then $\Delta(C_i) \le r$ and $\Z(C_i) \le (r-1)\alpha(C_i)$ by the inductive hypothesis. 
    For each $C_i$ that is not isomorphic to $K_{r+1}$, let $B_i$ be a zero forcing set of $C_i$. 
    If $C_i$ is isomorphic to $K_{r+1}$, then there exist $x_i, y_i \in V(C_i)$ such that $x_i$ and $y_i$ do not have a common neighbor in $A$; otherwise, $G$ is disconnected or complete. 
    For each $C_i$ that is isomorphic to $K_{r+1}$, let $B_i = V(C_i) \setminus \{x_i, y_i\}$. 
    Therefore, $|B_i| \le (r-1) \alpha(C_i)$ for all $i$.

    Let $B = A \cup B_1 \cup \cdots \cup B_k$. 
    It is straightforward to show that $B$ is a zero forcing set of $G$, since each component of $H$ not isomorphic to $K_{r+1}$ can be forced blue and each $x_i$ can be forced by a vertex in $A$, allowing for the remaining components of $H$ to be forced completely blue (similar to the base case, no vertex in $A$ can be adjacent to two vertices of degree $r$ in different components of $H$). 
    Therefore, $\Z(G) \le |B| = \alpha(G) + \displaystyle\sum_{i=1}^k |B_i| \le \alpha(G) +(r-1)\sum_{i=1}^k \alpha(C_i).$ 
    Since the union of the independent sets from each $C_i$ is an independent set in $G$, $\displaystyle\sum_{i=1}^k \alpha(C_i) \le \alpha(G)$. 
    Hence, $|B| \le r\alpha(G)$.\end{proof}

\subsection*{Acknowledgements}
The authors would like to acknowledge Randy Davila, who first conjectured that $3$-$1$ trees would provide a way of constructing an infinite family of graphs witnessing the tightness of the bound in Conjecture \ref{theconjecture}. 
 Conjectures generated by TxGraffiti can be found in a useful interface at \cite{txgraffiti}. Thanks also to Markus Meringer for providing $3$-regular graph data that allowed for computational testing of many early conjectures.

\end{document}